\documentclass[preprint]{elsarticle}
\usepackage{amsmath, amssymb,amsfonts}
\usepackage{graphicx}
\usepackage{amscd}
\usepackage{amsthm}

\biboptions{longnamesfirst,semicolon}

\newcommand{\Lt}{\mathcal{L}}
\newcommand{\Lnt}{\mathcal{L}_n}

\newcommand{\Drs}{D_{r}X_{s}}
\newcommand{\Drt}{D_{r}X_{t}}

\newcommand{\Dnrt}{D_{r}X^n_{t}}

\newcommand{\Dnrs}{D_{r}X^n_{s}}

\newcommand{\Dnt}{DX^{n}_{t}}

\newcommand{\bnps}{ b'_{n}(X^{n}_{s})}

\newcommand{\fnps}{ f'(X^{n}_{s})}

\newcommand{\R}{\mathbb{R}}

\newcommand{\E}{\mathbb{E}}
\newcommand{\D}{\mathbb{D}}

\numberwithin{equation}{section}

\newtheorem{thm}{Theorem}[section]
\newtheorem{lem}[thm]{Lemma}

\newtheorem{cor}[thm]{Corollary}

\newtheorem{rem}{Remark}

\newtheorem{hypo}[thm]{Hypothesis}

\newtheorem{prop}[thm]{Proposition}

\title{Smooth density for the Solution of Scalar SDEs with Locally Lipschitz Coefficients under H\"ormander Condition}
\author{M. Tahmasebi }
\ead{mh_tahmasebi2000@yahoo.com} 

\address{
Department of Applied Mathematics, Faculty of Mathematical Sciences,
Tarbiat Modares University,
P.O. Box 14115-134,  Tehran, Iran}

\begin{document}
\begin{abstract}
In this paper the existence of a smooth density is proved for the solution of an SDE, with locally Lipschitz coefficients and semi-monotone drift, under H\"ormander condition. We prove the nondegeneracy condition for the solution of the SDE, from it  an integration by parts formula would result  in the Wiener space. To this end we construct a sequence of SDEs with globally Lipschitz coefficients whose solutions converge to the original one and use some Lyapunov functions  to show the uniformly boundedness of  the $p$-moments of the solutions and their Malliavin derivatives.
\end{abstract}
\maketitle

\section{Introduction}
In this paper, we study Malliavin differentiability and the existence of a smooth density for the solution of an SDE. 
We consider a scalar SDE whose semi-monotone drift and locally Lipschitz coefficients satisfy  H\"ormander condition. Such equations are considered mostly in finance, biology, and  dynamical systems and are more challenging when considered on infinite dimensional spaces. (see e.g. \cite{B99, Z95, GM07}).\\
We prove the existence of a unique, infinitely Malliavin differentiable, strong solution to this SDE satisfying some nondegeneracy condition, and derive both the integration by parts formula in the Wiener space and the existence of a smooth density for this solution.\\
This subject have been studied by many authors, mostly in the case where the coefficients are globally Lipschitz. In \cite{Kusu84}
Kusuoka and Stroock have shown that an SDE whose coefficients are $C^{\infty}$-globally Lipschitz with polynomial growth, has a strong Malliavin differentiable solution of any order. The absolute continuity of the law of the solution of SDEs with respect to the Lebesgue measure and the smoothness of its density under some nondegeneracy condition are shown in \cite{Nualart06, Jacod87}. Nualart (2006) shows that the H\"ormander condition,  posed on the coefficients,  condition  {\bf (H)} in this paper, implies this required nondegeneracy condition, if the coefficients are $C^{\infty}$-globally Lipschitz and have polynomial growth.  
Assuming the nondegeneracy condition one can derive some integration by parts formula in the Wiener space and also  reqularity for the distribution of the solution (see e.g.  \cite{Nualart06}). 
It is often of interest to investors to derive an option pricing formula and to know its sensitivity with respect to various parameters. The integration by parts formula obtained from Malliavin calculus can transform the derivative of the option price into a weigthed integral of random variables. This gives much more accurate and fast converging numerical solutions  than obtained by the classical methods (\cite{Higa04, Bavou06}). 
The interested reader could see \cite{Alos08, Marco11}).\\
 In recent years, there were attemps to generalize these results to SDEs with non-globally Lipschitz coefficients.  
For example,  in \cite{Fourier10} using a  Fourier transform argument,  some absolute continuity results are obtained for the law of the solution of an SDE with H\"older coefficients.
The existence of densities for a general class of non-Markov It\"o processes under some spatial ellipticity condition and 
 that allow the degeneracy of the diffusion coefficient is shown in \cite{Bell04}.  
Marco \cite{Marco11} has shown that assuming some local properties of coefficients,  and uniform ellipticity of diffusion coefficient, the law of the solution of the SDE has smooth density. If the diffusion coefficient is uniformly elliptic, then  the H\"ormander condition is satisfied. When the noise is a fractional brownian motion or a Levy process the same results are obtained under ellipticity and H\"ormander condition as well. 
For other references on this subject, we
refer the reader to \cite{Kusuo10, Marco10, Hiraba92, Baudoin11}. \\
To deal with the SDE with non-globally Lipschitz coefficients, we construct a sequence of SDEs with globally Lipschitz coefficients whose solutions are Malliavin differentiable of any order and satisfy a nondegeneracy condition. In this way we can apply the classical Malliavin calculus to the solutions. We can find also a uniform bound for the moments of the solutions, and all their Malliavin derivatives, by using some Lyapunov functions. Then we will prove the nondegeneracy condition for the original SDE, using the nondegeneracy condition for the sequence of solutions to the constructed SDEs.
This result implies the integration by parts formula in the Wiener space and the existence of the smooth density for the solution.\\
The paper is organized as follows. In section 2, we recall some basic results from Malliavin calculus that will be used in the paper, in particular the integration by parts formula due to \cite[Proposition 2.1.4]{Nualart06}. In section 3, we state the assumptions and our main results; In section 4, we prove the uniformly boundedness for the moments of Malliavin derivatives of the solution to a sequence of approximating SDEs, as there exist some Lyapunov functions.  Section 5  involve the construction of our approximating SDEs with globally Lipschitz coefficients, and proving the convergence of their solutions and their Malliavin derivatives to those of the solution of the original SDE (\ref{equa}). Also, we introduce some Lyapunov functions which results to the infinitely weak differentiability. In section 6, we will prove the nondegeneracy condition that implies the integration by parts formula and the existence of smooth density. Finally, in Appendix we state the detailed proof on selection of approximating processes.
\section{Some basic results from Malliavin calculus}\label{ma}
In this article, we use the same notations as in \cite{Nualart06}. Let $\Omega$ denote the Wiener space
$C_{_{0}}([0,T];\R)$ endowed with
$\parallel.\parallel_{_{\infty}}$-norm making it a (separable)
Banach space. Consider a complete
probability space $(\Omega,\mathcal{F},P)$, in which $\mathcal{F}$ is generated by the open
subsets of the Banach space $\Omega$, $W_{t}$ is a d-dimensional Brownian motion, and $\mathcal{F}_{t}$ is the filtration generated by $W_t$. \\
Consider the Hilbert space $H:=L^2([0,T];\R)$.
The Malliavin derivartive operator  
$D$ is closable from $L^p(\Omega)$ to $L^p(\Omega, H)$, for every $p \geq 1$ and the adjoint of the operator $D$ is denoted by $\delta$. 
We use the notation $\bigwedge_F=\Vert DF \Vert_H^2$
to show the Malliavin
covariance matrix for a random variable $F$, and  for every $k \geq 1$, we set $D_{r_1 \cdots r_k}F = D_{r_k}(D_{r_1 \cdots r_{k-1}}F)$.\\
Now let $Y_t$ be a solution to the following SDE;
\begin{equation}\label{maequa}
dY_{t}=B(Y_t)dt+A(Y_t)dW_{t}    \qquad Y_0=x_{0},
\end{equation}
where  $B:\R \longrightarrow \R$  is a measurable function and  $A:\R \longrightarrow \R$ is an $C^\infty$ function. 
Let $Z_t$ be the solution of the following linear SDE;
$$Z_t= 1+ \int_{0}^{t} B'(Y_s) Z_s ds + \int_{0}^{t}  A'(Y_s)Z_s dW_s $$
and 
\begin{equation*}
C_t:= \int_{0}^{t} (Z_s^{-1} A(Y_s))^2 ds, 
\end{equation*}
and assume the  H\"{o}rmander's condition holds as follows:\\
{\bf (H)} ~ ~ $A(x_0) \neq 0$ or $A^{(n)}(x_0)B(x_0) \neq 0$ for some $n \geq 1$.\\
Under this condition Nualart  \cite{Nualart06} 
has shown the following proposition.
\begin{prop}\label{horman}
For a solution $Y_t$ to an SDE with globally Lipschitz coefficients and
polynomial growth for all their derivatives, the H\"{o}rmander's condition {\bf (H)} implies that for any $p \geq 2$ and any $\epsilon$ small enough,
\begin{equation}\label{chor}
P\Big(C_t \leq \epsilon\Big) \leq \epsilon^p
\end{equation}
and $(detC_t)^{-1} \in L_p(\Omega)$ for all $p$.Thereby obtaining the nondegeneracy condition for $Y_t$ and thus the integration by parts formula in the Wiener space and an infinitely differentiable density, too. 
\end{prop}
\section{Formulation of main results}
In this section we consider the $\mathcal{F}_{t}$-adapted stochastic process $X_{t}$  taking values in $\R$,  which is a solution to the following stochastic differntial equation\\
\begin{equation}\label{equa}
dX_{t}=[b(X_{t})+f(X_{t})] dt+\sigma(X_{t})dW_{t},  \qquad X_0=x_0.
\end{equation}
where  $b,f:\R \longrightarrow \R$ are measurable functions and  $\sigma:\R \longrightarrow \R$ is a measurable $C^\infty$ function. 
We denote by $\Lt$ the second-order differential operator associated to SDE (\ref{equa}):
\begin{equation*}
\Lt= \frac{1}{2} \sigma^2(x)  \partial^2 +  [b(x)+f(x)]\partial
\end{equation*}
Throughout the paper we assume that $b$, $f$ and $\sigma$ satisfy the following Hypothesis.
\begin{hypo}\label{hypo1}
Functions $b,f$ and $\sigma$ satisfiy the following conditions.
\begin{itemize}
 \item  
 $b$ and $\sigma$ are $C^\infty$ 
 locally Lipschitz and all of their derivatives have polynomial growth; i.e., for every $j \geq 0$ there exist some constants $\lambda_j$ and $q_j$  such that for every $x \in \R$
 \begin{equation}\label{polyno}
\vert b^{(j)} (x) \vert +  \vert \sigma^{(j)} (x) \vert \leq \lambda_j(1+\vert x \vert^{q_j})
 \end{equation}  
 Also, set $\xi:= \max_{j \geq 1} q_j < \infty$.
 \item  The function $f$ is $C^\infty$, globally Lipschitz  with the Lipschitz constant $k_1$ and all of its derivatives  are bounded. 
\end{itemize}
\end{hypo}
Since the coefficients are locally Lipschitz functions, we need further assumptions that the solution of (\ref{equa}) does not explode. Also, to obtain the Malliavin differentiability of the solution $X_t$ in all $t \in [0,T]$, we consider the following hypothesis.   
\begin{hypo}\label{hypo2}
\begin{itemize}
\item  The function $b$ is a $C^\infty$ uniformly monotone function, i.e., there exists a constant $K$ such that 
 for every $x,y \in \R$,
\begin{equation}\label{monoton}
\Big(b(y)-b(x)\Big)(y-x) \leq -K{\vert y-x \vert}^2.
\end{equation}
 \item For every $p \geq 1$ there exist some constants $\alpha_p$ and $\beta_p$ such that
 \begin{equation}\label{Bound}
a b(a)+(12p-1) \sigma^2(a) + (4p-1)  \vert a\sigma'(a)\vert^2  \leq \alpha_p + \beta_p
\vert a \vert^2   \qquad  \forall a \in \R,
\end{equation}
\end{itemize}
\end{hypo}
\begin{rem}
Notice that for every monotone drift $b$ and every globally Lipschitz diffusion $\sigma$, Hypothesis \ref{hypo2} is satisfied. Hypothesis \ref{hypo2} is satisfied also for some locally Lipschitz diffusion $\sigma$, for example when $b(x)=-x^5$ and $\sigma(x)=x^2$.
\end{rem}
Since the constant $\beta_p$ in (\ref{Bound}) could be negative, we may use the version of Gronwall's inequality which is proved in \cite[Lemma 1.1]{Hasminskii12}.\\
Now we are going to present the higher order differentiability and the integration by parts formula for the solution of SDE(\ref{equa}) in the Wiener space.   
\begin{thm}\label{Dinf}
Under Hypotheses \ref{hypo1} and \ref{hypo2}, the SDE (\ref{equa}) has a unique strong solution in $\D^{\infty}$.
\end{thm}

We also show the 
nondegeneracy condition for $X_t$ which derive the integration by parts formula in \cite[Proposition 2.1.4]{Nualart06} and thus an smooth density.
\begin{thm}\label{DENSF}
The solution $X_t$ of SDE (\ref{equa}) is nondegenerate and has smooth density.
\end{thm}
\section{Malliavin differentiability and Lyapunov functions}
In this section, we consider the SDE (\ref{equa}) with Hypothesis \ref{hypo1} and introduce some new stochastic sytems by using a sequence of $\{X_t^n\}$ and prove that if there exist some suitable Lyapunov functions for these systems, then the SDE has an infinitely Malliavin differentiable solution. It will be usefull for the next section, as we construct desired Lyapunov functions under Hypothesis \ref{hypo2}.\\ 
Consider two sequences of functions $b_n, \sigma_n \in C_{loc}^1$ and $G_n \subseteq \R$ such that $\bigcup G_n=\R$ and for every $x \in G_n$,  $\sigma_n(x)= \sigma(x)$ and  $b_n(x) =b(x)$. Assume that there exists a Lyapunov function $F$; ( i.e.,   F is positive and for some constant $p \geq 2$ and every $x$ we have $F(x) \geq c_0 \vert x \vert^{p}$)
     and a sequence of strong solutions $\{X_t^n\}\subseteq \D^{1, p}$ to the SDE's 
\begin{equation*}
X_{t}^n=x_{0}+\int_{0}^{t} b_{n}(X_{s}^n)ds+\int_{0}^{t}
\sigma_n(X_{s}^n)dW_{s}
\end{equation*}
such that 
\begin{equation}\label{supF}
\sup_{n \geq 1}\sup_{0 \leq t \leq T}\E\Big[ F(X_t^n)  \Big] < \infty.
\end{equation}
Let $\tau_n$ is the first exit time of $X_t$ from $G_n$. Also, assume that the sequence $X_t^n$ converges almost surely, or for the transition probability measure  $P^{(n)}$ associated with $X^n$,
\begin{equation}\label{pn}
 \lim_{n \longrightarrow \infty}  P^{(n)}(\tau_n \leq t) =0
\end{equation}
Our main result is the following theorem which proves the Malliavin differentiability for the solution of an SDE,  using Lyapunov function. 
\begin{thm}\label{2lay}
Assume that there exists some Lyapunov function $V(.,.)$ and functions $K(.,.)$ and $G(.)$ for the linearized system 
\begin{equation*}
LS(r):=\left\{ \begin{array}{lcr}
X_{t}^n=x_{0}+\int_{0}^{t} \Big(b_{n}(X_{s}^n)+f(X_s^n)\Big)ds+\int_{0}^{t}
\sigma_n(X_{s}^n)dW_{s}  \\
\\
Y^n_t =\sigma_n(X^n_{r})+\int_{r}^{t}\Big(\bnps+\fnps\Big)Y^n_s
ds+\int_{r}^{t} \sigma'_n(X_{s}^n)Y^n_s dW_{s}, 
\end{array}  \right.
\end{equation*}
with infinitesimal operator $L_{n,r}$ such that for some constant $c$  
\begin{equation}\label{L}
L_{n,r}V(x,y) \leq c \Big( V(x,y)+G(x,y)+h(x) \Big),
\end{equation}
\begin{equation}\label{G}
\E\Big[G(X_t^n, \Dnrt)\Big]=0,~   \forall t > r, \qquad \textmd{and}  \qquad 
\sup_{n \geq 1}\sup_{0 \leq t \leq T}\E\Big[h(X_t^n)\Big] < c,   
\end{equation}
\begin{equation}\label{cL}
\Big(\partial_x V(x,y) \sigma_n(x) + \partial_y V(x,y) y\sigma'_n(x)\Big)^2 \leq c K(x,y),
\end{equation}
\begin{equation}\label{r}
\sup_{n \geq 1}\E\Big[V(X_r^n, \sigma_n(X_r^n)\Big] \leq c, \qquad \textmd{and}  \qquad 
\sup_{r \leq t \leq T}\E\Big[K(X_t^n, D_rX_t^n)\Big] < \infty.   
\end{equation}
Then there exists a random process $X_t \in \D^{1, p}$ which is the solution of SDE (\ref{equa}),  
\begin{equation*}
X_t^n \longrightarrow X_t  ~ in ~ L^p([0,T] \times \Omega),
 \qquad \textmd{and}  \qquad 
 DX^n \longrightarrow DX \vspace{0.25cm}~ in ~ \D^{1,p}. 
\end{equation*}
\end{thm}
\begin{proof}
By  It\^o's formula, one can easily show that
\begin{equation*}
\sup_{n \geq 1}\sup_{r \leq t \leq T} \E\Big[ V(X_{t}^n, \Dnrt) \Big]   < \infty. 
\end{equation*}
Therefore, there exists a subsequence $(X_{t}^{n_k}, DX_t^{n_k})$ convergent weakly in $L^p(\Omega)\otimes L^p(\Omega; H)$ to some $(X_t, Y_t)$. By assumptions,
 if the sequence $X_t^n$ converges almost surely, then $X_t^{n_k} \longrightarrow X_t$ in $L^p(\Omega)$, and therefore  $\{X_t^{n_k}\}$ is uniformly integrable. Since the inequality (\ref{pn}) holds, by Corollary 10.1.2 and Theorem 10.1.3 in \cite{Stroock79} there exists a unique measure $P_0$ solving the  the martingle problem for $b+f$ and $\sigma$. Also, because this martingale problem is well-posed, there exists a unique weak solution $X_t$ for the SDE (\ref{equa}) which is  the strong solution which we assumed the SDE (\ref{equa}) has and  $P_0=P$. Now, according to Theorem 4.5.2. in \cite{Chung01}, we have a sequence $\{X_t^{n_k}\}$ which converges to $X_t$ in $L^p$  and $\sup_{n \geq 1}\E(\Vert \Dnrt \Vert_H^p) < \infty $. Therefore, by Lemma 1.2.3 in \cite{Nualart06}, $X_t \in \D^{1,p}$ and $DX_t^{n_k}\longrightarrow DX_t$ in $L^p(\Omega; H)$.
\end{proof}
For every $m \geq 2$ and $r=(r_1,\cdots, r_m)$ with $r_1 < r_2 < \cdots < r_m$,  let $kSm$ be the set of all permutations of $(r_1, \cdots, r_m)$ of the form $R=(R_{l_1},  \cdots,R_{l_k})$ where $l_0=0$ and
$R_{l_j}:=r_{i_{_{l_{j-1}+1}}}, \cdots, r_{i_{l_j}}$, such that $i_{l_{j-1}+1} < \cdots < i_{l_{j}}$ for every $1 \leq j \leq k$. Set 
$D_R X_t^n:=D_{R_{l_1}}X_t^n \cdots D_{R_{l_k}}X_t^n$ and consider the linearized system $LS(r_1, \cdots, r_m)$ defined by the form
\begin{equation}
\left\{ \begin{array}{lcr} 
 \displaystyle{\bigcup_{\substack{ kSm, \\
  2 \leq k \leq m}}}~ LS(r_{i_1}, \cdots, r_{i_k})    \qquad  0 \leq r_1 < \cdots < r_m \leq t\\
 \\
\begin{aligned} 
dD_{r_1 \cdots r_m} & X_t^n  = 
 \displaystyle{\sum_{\substack
{ kSm, \\
  1 \leq k \leq m}} } b_n^{(k)}(X_t^n) D_RX_t^n dt  +  \displaystyle{\sum_{\substack{ kSm, \\
  1 \leq k \leq m}}}  \sigma_n^{(k)}(X_t^n) D_RX_t^n dW_t , 
\end{aligned}
\end{array}  \right.
\end{equation}
\\
with infinitesimal operator $L_{n,(r_1, \cdots, r_m)}(x,y)$,  $y \in \R^{N_m}$, where $N_m= \sum_{1 \leq k \leq m}\sharp(kSm)$ and $\sharp(kSm)$ is the number of elements of  $kSm$. For the sake of simplicity set  
\begin{equation}
 A_{m,n} (X_t^n, Y_{t,m-1}^n) := \displaystyle{\sum_{\substack
{ kSm, \\
  2 \leq k \leq m}} } \sigma_n^{(k)}(X_t^n) D_R X_t^n.
\end{equation}
Here $Y_{t,m-1}^n$ is a random vector in $\R^{N_m-1}$ which components are  the  Malliavin derivatives of $X_t^n$ up to the order $m-1$, appeared in the system $LS(r_1, \dots, r_m)$. \\
The diffusion coefficient in this system will be denoted by 
\begin{equation}
\sigma_{m,n}(X_t^n)=\begin{pmatrix}
\sigma_{m,n}^0(X_t^n) \\
 \sigma'_n(X_t^n) D_{r_1 \cdots r_m} X_t^n + A_{m,n}(X_t^n, Y_{t,m-1}^n)
\end{pmatrix}
\end{equation}
Our next result states  that if  a Lyapunov function $V(.)$ is constructed for the system $LS(r)$, then for every $m \geq 2$ another one could be constructed  for the linearized system $LS(r_1, \dots, r_m)$.  
\begin{thm}\label{mlay}
Suppose that the assumptions in Theorem \ref{2lay}  hold, and  for each $y \in \R^{N_m}$, define the Lyapunov function $V_m(x, y, z)=V(x,z)$. If for some constant $c_0$ the following inequalities hold:
\begin{equation}\label{cLm}
\vert \partial_{xz} V_m(x,y,z) \vert^2+\vert \partial_z V_m(x,y,z) \vert^2+ \vert \partial_{zz} V_m(x,y,z) \vert^2 \leq c_0 V(x,z) + c_0F(x) +c_0 ,
\end{equation}
\begin{equation}\label{rm}
\sup_{n \geq 1}\E\Big[V_m(X_{r_m}^n, Y_{r_m,m-1}^n , \sigma^0_{m,n}(X_{r_m}^n))\Big] \leq c_m,
\end{equation}
\begin{equation}\label{xym}
\sup_{n \geq 1}\sup_{0 \leq t \leq T}\E\Big[ \vert X_t^n \vert^q + \vert Y_{t,m-1}^n \vert^q  \Big] < c_0,  \qquad  {1 \leq q \leq 2(m+1)},
\end{equation}
then 
\begin{equation*}
\sup_{n \geq 1}\sup_{r \leq t \leq T} \E\Big[ V_m(X_{t}^n, Y_{t,m-1}^n, D_{r_1, \cdots, r_m}X_t^n) \Big]   < \infty 
\end{equation*}
\end{thm}
\begin{proof}
According to Theorem \ref{2lay}, it is sufficient to show that
for every $m \geq 1$, there exist some function $F_m$ and a constant $c_m$ such that 
\begin{equation}\label{Lm}
L_{n, (r_1, \dots, r_m)} V_m(x, y, z) \leq c_m\Big( V_m(x, y, z)+G(x)+F_m(x,y) \Big),
\end{equation}
\begin{equation}\label{Fm}
\sup_{n \geq 1} \sup_{0 \leq t \leq T}\E[F_m(X_t^n,Y_{t,m-1}^n )] < \infty ,
\end{equation}
\begin{equation}\label{Km}
\sup_{0 \leq t \leq T}\E\Big[\nabla V_m(X_t^n, Y_{t,m-1}^n, D_{r-1 \cdots r_m}X_t^n).\sigma_{m,n}(X_t^n)\Big] < \infty.   
\end{equation}
We have \\
$\begin{aligned}
L_{n, (r_1, \dots, r_m)} V_m(x,y,z) &= L_{n,r_1}V(x,z) + \partial_z V(x,z) \Big( B_{m,n}(x,y) +A_{m,n}(x,y) \Big) \\
 & + \partial_{xz} V(x,z)   A_{m,n}(x,z)+ \frac12   \partial_{zz} V(x,z)\Big(A_{m,n}(x,z) \Big)^2  \\
& \leq cV(x,z) +cG(x)+cF(x) +( \partial_{z} V(x,z) )^2+( \partial_{xz} V(x,z) )^2  \\
& +  \frac14  ( \partial_{zz} V(x,z) ) ^2 + \Big(B_{m,n}(x,y) \Big)^2+\Big(A_{m,n}(x,y) \Big)^2 + \Big(A_{m,n}(x,y) \Big)^4, 
\end{aligned}$\\
where in the last inequality we have used (\ref{L}). Let 
$$ F_m(x,y) := (c+c_0)F(x) + \Big(B_{m,n}(x,y) \Big)^2+\Big(A_{m,n}(x,y) \Big)^2 + \Big(A_{m,n}(x,y) \Big)^4 +c_0,  $$
applying (\ref{cLm}), we get
\begin{equation}
L_{n, (r_1, \dots, r_m)} V_m(x,y,z) \leq (c+3c_0)V_m(x,y,z) + cG(x) + F_m(x,y).
\end{equation}
Since all  of the derivatives of $b$ and $\sigma$ have polynomial growths, (\ref{xym}) is deduced from  inequality (\ref{Fm}).  
By definition of the function $V_m$, we know that 
$$\nabla V_m(x,y,z).\sigma_{m,n}(x)= \partial_x V(x,z) \sigma_n(x) + \partial_z V(x,z) z \sigma'_n(x) +  \partial_z V(x,z) A_{m,n}(x,y),$$
using (\ref{xym}) and (\ref{cLm}) once again we conclude that 
\begin{equation*}
\sup_{0 \leq t \leq T}\E\Big[\partial_z V_m(X_t^n, D_{r_1 \cdots r_m}X_t^n)A_{m,n}(X_t^n, Y_{t,m-1}^n )\Big] < \infty.   
\end{equation*}
By (\ref{r}) and (\ref{cL}),  inequality (\ref{Km}) results and the proof is completed.
\end{proof}
By a bit more calculations we can  derive the same result   when  $(r_1, \cdots, r_m)$ has no increasing components. 
\begin{cor}
If the assumptions in Theorems \ref{2lay} and \ref{mlay} hold for every integer $m \geq 1$ and $p,q \geq 1$, then $X_t \in \D^{\infty}$
\end{cor}
It is worth noting  that the required conditions are not too restrictive. In the next section we will show that by Hypothesis \ref{hypo2}, they are satisfied and the solution of (\ref{equa}) is Malliavin differentiable of any order. 
\section{Malliavin Differentiability under Hypothesis \ref{hypo2}}
It is well-known that under Hypothesis \ref{hypo2} the SDE (\ref{equa}) has a strong solution $\{X_t\}$ \cite{Mao97}. The uniqueness of the solution is obtained by using It\^o's formula and Gronwall's inequality. In this section, we will show that this solution is in $\D^{\infty}$. To this end, we first show that $X_t \in L^P(\Omega)$,  does not  explode in finite time, and the process $\sup_{0 \leq s \leq t} X_s$ has bounded moments. Then, we construct an almost everywhere convergent sequence of processes $X^n_t$ whose limit is $X_t$. For every $p \geq 2$, we will find some function $V_p$ such that conditions (\ref{supF}), (\ref{L}) and (\ref{r}) hold, so that  $X_t \in \D^{1,p}$. We construct Lyapunov functions with  polynomial growths thus satisfying (\ref{rm}) and conclude that $X_t \in \D^{\infty}$. This procedure is followed in the next subsections.\\
In what follows  $G_n\subseteq\R$, is set 
 \begin{equation}
 G_n=\Big\{ x \in \R; \quad  \vert x \vert \geq n^\xi \Big\}
 \end{equation}
for each $n \geq 1$, and $\tau_{n}$ is the first exit time of $X_t$ from $G_n$. 
\subsection{Some Properties of the solution}

\begin{lem}\label{xlp}
For each $t \in [0,T]$ and $p > 1$, $X_{t}$ belongs to $L^p(\Omega)$ and does not explode in finite time.
\end{lem}

\begin{proof}
Using Fatou's lemma, we first show that  
$X_t$ is in $L^p(\Omega)$. By  definition of the operator $\Lt$ and inequality (\ref{Bound}), we have

$\begin{aligned}
\Lt X_t^{2p} & = 2p X_t^{2p-1}\Big(b(X_t)+f(X_t) \Big)+ p(2p-1)\sigma^2(X_t) X_t^{2p-2}\\
    &
  \leq 2p\beta_p X_t^{2p}+ 2p\alpha_p X_t^{2p-2} + 2p X_t^{2p-2} X_t f(X_t) \\
  & \leq 2p\beta_p X_t^{2p}+ 2p\alpha_p X_t^{2p-2} + 2p X_t^{2p-2} \Big( \frac{X_t^2}{2}+ \frac{f^2(X_t)}{2}  \Big) \\
  & \leq p(2\beta_p+2k_1^2+1) X_t^{2p} + 2p(\alpha_p+\vert f(0) \vert^2) X_t^{2p-2} =: \beta'_p X_t^{2p} + \alpha'_p  X_t^{2p-2},
\end{aligned}$\\
\\
where in the last inequality we used the Lipschitz property of $f$.
Applying It\^o's formula,

\begin{equation}\label{induc}
\frac{d}{dt} \E\Big[X_{t \wedge \tau_n}^{2p} \Big] = \E\Big[ \Lt X_{t \wedge \tau_n}^{2p} \Big]  \leq   \beta'_p \E\Big[X_{t \wedge \tau_n}^{2p} \Big] + \alpha'_p \E\Big[X_{t \wedge \tau_n}^{2p-2} \Big].
\end{equation}
Setting $p=1$ and using Gronwall's inequality in \cite{Hasminskii12}, we derive

\begin{equation*}
 \E\Big[X_{t \wedge \tau_n}^{2} \Big] \leq  \Big(\vert x_0 \vert^2+ \frac{\alpha_1}{\beta_1}\Big)exp\{3\beta_1 T\}-\frac{\alpha_1}{\beta_1}
\end{equation*}
By (\ref{polyno}), we have 
\begin{equation*}
\Big({\frac{n}{2}-1}\Big)^{\frac{1}{q_0+1}} P\Big(t \geq \tau_n\Big)  \leq   \Big(\vert x_0 \vert^2+ \frac{\alpha_1}{\beta_1}\Big)exp\{3\beta_1 T\}-\frac{\alpha_1}{\beta_1}
\end{equation*}
Letting $n$ tend to $\infty$, we conclude that $lim_{n \rightarrow \infty}\tau_{n}=\infty$ almost surely and thus by Fatou's lemma 
\begin{center}
$\E[X_t^2]  \leq \E\big[\displaystyle \liminf_{n \rightarrow \infty} X_{t \wedge \tau_{n}}^2\big]
                        \leq  \displaystyle \liminf_{n \rightarrow \infty} \E\big[ X_{t \wedge \tau_{n}}^2\big]
                        \leq \Big(x_0^2+ \frac{\alpha_1}{\beta_1}\Big)exp\{3\beta_1 T\}-\frac{\alpha_1}{\beta_1}.$
\end{center}
Hence from (\ref{induc}) and induction on $p$ we get $X_{t}\in L^{2p}(\Omega)$. Now by the following interpolation inequality 
$$ \E\Big[\vert X_t \vert^{2p+1}\Big] \leq \Big(\E[X_t^{2p}]\Big)^{\frac12} \Big(\E[X_t^{2p+2}]\Big)^{\frac12},$$ 
we conclude that  $X_{t}\in L^p(\Omega)$ for every $p > 1$.
\end{proof}

In the proof of the next Lemma we will use the following version of the Young's inequality. 
For $r \geq 2$ and for all $a,c$ and $\triangle_1 >0$, we have:
\begin{equation}\label{young}
a^{r-2}c^2 \leq \triangle_1^2\frac{r-2}{r}a^r+\frac{2}{r\triangle_1^{r-2}}c^r.
\end{equation}

\begin{lem}\label{io}
 For every $p \geq 2$, if $\beta_p \geq 0$, there exists some constant $c'_p$ such that
\begin{equation}
\E\Big[\sup_{0 \leq t \leq T}\vert X_t \vert^p \Big] < c'_p
\end{equation}
\end{lem}
\begin{proof} 
We know that for every $p \geq 2$, $\E[\vert X_t \vert^p] < \infty$. Applying It\"o's formula and the Burkholder-Davis-Gundy inequality in \cite{Ren08}, we have \\
\begin{align}
\E\Big[\sup_{0 \leq t \leq t_1}\vert X_t \vert^{2p}\Big] & =\E[x_0^{2p}]+\E\Big[\sup_{0 \leq t \leq t_1}\int_{0}^{t} \Lt X_s^{2p}ds\Big]  +2p\E\Big[\sup_{0 \leq t \leq t_1}\int_0^t X_s^{2p-1}\sigma(X_s)dW_s\Big]  \nonumber\\
& \leq \E[x_0^{2p}]+\E\Big[\sup_{0 \leq t \leq t_1}\int_{0}^{t} \Lt X_s^{2p}ds\Big]+6p\E\Big[\sup_{0 \leq t \leq t_1}\vert \int_{0}^{t} \vert X_s \vert^{4p-2}\sigma^2(X_s) ds\vert^\frac12\Big]   \nonumber\\
 & \leq \E[x_0^{2p}]+\E\Big[\sup_{0 \leq t \leq t_1}\int_{0}^{t} \Lt X_s^{2p}ds\Big]  \nonumber\\
 &+6p\E\Big[\sup_{0 \leq t \leq t_1}\vert X_s \vert^{p} \sup_{0 \leq t \leq t_1}\vert \int_{0}^{t} \vert X_s \vert^{2p-2}\sigma^2(X_s) ds\vert^\frac12\Big]  \nonumber\\
& \leq \E[x_0^{2p}]+\E\Big[\sup_{0 \leq t \leq t_1}\int_{0}^{t} \Lt X_s^{2p}ds]+\frac{1}{2}\E[\sup_{0 \leq t \leq t_1}\vert X_s \vert^{2p}\Big]   \nonumber\\
&+\frac{36p^2}{2}\E\Big[ \sup_{0 \leq t \leq t_1} \int_{0}^{t} \vert X_s \vert^{2p-2}\sigma^2(X_s) ds\Big], \label{supc}
\end{align}
where the last inequality holds by   $ac \leq \frac{L}{2}a^2 + \frac{1}{2L}c^2$ for $L=6p$. 
By inequality (\ref{Bound}) in (\ref{supc}), we derive
\begin{equation}\label{2pc}
(1-\frac{1}{2C})\E\Big[\sup_{0 \leq t \leq t_1}\vert X_t \vert^{2p}\Big] \leq \E[x_0^{2p}]+\E\Big[\sup_{0 \leq t \leq t_1}\int_{0}^{t} [\alpha_p\vert X_s \vert^{2p-2}  +\beta_p \vert X_s \vert^{2p}] ds\Big].  
\end{equation}
Using (\ref{young}) for $r=2p$ with $c=\triangle_1=1$, we have 
$x^{2p-2} \leq \frac{p-1}{p}x^{2p}+\frac{1}{p}.$
Substituting this latter inequality in (\ref{2pc}), we  deduce that
\begin{align}
(1-\frac{1}{2C})\E\Big[\sup_{0 \leq t \leq t_1}\vert X_t \vert^{2p}\Big] &  \leq \E[x_0^{2p}]+\frac{\alpha_p}{p}t_1+\int_{0}^{t_1}(\alpha_p\frac{p-1}{p}+\beta_p)\E\Big[\sup_{0 \leq s \leq t_1}\vert X_s \vert^{2p}\Big]ds,  
\end{align}
and Gronwall's inequality completes the proof.
\end{proof}

\subsection{Approximation of the solution}
For every integer $n >0$ let choose some smooth functions $\varphi_{n}:\R \rightarrow \R$ such that $\phi_{n}=1$ on $A_{n}:=\{x \in \R;~\mid
x\mid \leq n^{\xi} \}$ and $\phi_{n}=0$ outside  $A_{2n^{\xi}}$. Also for each multiindex $L$ with $\vert L \vert =l\geq 1$,
\begin{equation}\label{supderbphi}
\sup_{_{n,x}}\Big( \vert \partial_{_{L}} \phi_n \vert +\vert \langle b , \partial_{L} \phi_{n} \rangle \vert +\vert \sigma \partial_{L} \phi_{n} \vert \Big)\leq M_{l}
\end{equation}
 for some $M_{l} > 0$. (See Appendix and the proof of Lemma 2.1.1 in \cite{Nualart06}).
  Now, set
$$b_{n}(x):=\phi_{n}(x)b(x) ~~ and  ~~ \sigma_n(x):= \phi_{n}(x)\sigma(x) $$ for every $x \in \R^d$ and  $n >0$. Then $b_{n}$ would be globally Lipschitz and continuously differentiable. By (\ref{polyno}) for each $x \in \R^d$ and each multiindex $L$ with $\vert L \vert = l \geq 1$, there exist positive constants $\Gamma_l$ and $p_l$ such that
 \begin{equation}\label{roshdbn}
 \vert \partial_L b_n(x)+ \partial_L \sigma_n(x) \vert^2 \leq \Gamma_l(1+\vert x \vert^{p_l}).
 \end{equation} Then, obviously, $b_{n}$ and $\sigma_n$ are continuously differentiable  and therefore globally Lipschitz. 
 By Hypothesis \ref{hypo2}, $b'(x)$ is negative and since $ac \leq a^2/2+c^2/2$ for all $a,c$ and $0 \leq \phi(.) \leq 1$, there exists a constant $C_{0,b}$ independent of $n$ such that
 \begin{equation}\label{bnprim}
  b'_n(x) \leq  C_{0,b},   \qquad  ~\textmd{and} ~\qquad  ( \sigma'_n(x))^2  \leq 2 (\sigma'(x))^2 +C_{0,\sigma} 
 \end{equation}
Notice that by Theorem 2.2.1 in Nualart (2006), the SDE (\ref{equan}) has a strong solution in $\D^{\infty}$, that is, there exists $X_t^n$ in $\D^{\infty}$ which satisfies 
\begin{equation}\label{equan}
X_{t}^n=x_{0}+\int_{0}^{t} b_{n}(X_{s}^n)ds+\int_{0}^{t}
\sigma_n(X_{s}^n)dW_{s}.
\end{equation}
Denote by $\Lnt$ the infinitesimal operator associated to the latter SDE.
\begin{lem}
For each $t \in [0,T]$ and $p > 1$, the sequence $\{X_{t}^n\}$ is uniformly
integrable and almost everywhere convergent to $X_{t}$.
\end{lem}

\begin{proof}
To prove the convergency, let $X^{\tau_{n}}$ denotes $X$ stopped at $\tau_{n}$.
By the proof of Lemma \ref{xlp}, $\tau_n$ increases to infinity as $n$ tends to infinity.
By the choice of the function $\phi_{n}(.)$, it then follows that $X_{t}^{\tau_{2n}}=X_{t}^{\tau_{n}}$ for all $t \leq \tau_{n}$.
Thus, for fix $t \in [0,T]$, letting $n$ tend to infinity, we
will have $lim_{n \rightarrow \infty}X_{t}^{n}=lim_{n \rightarrow \infty}X_{t}^{\tau_{n}}=X_{t}$ a.s.\\
Now, we prove that the sequence $\{X_{t}^n\}$ is uniformly
integrable. In fact, we show that for every $p > 1$,
\begin{equation}\label{ssup}
\sup_{n \geq 1}\sup_{0 \leq t \leq T}\E\Big[\vert X_{t}^n\vert^p\Big] \leq
c_{p}.
\end{equation}
By definition of the operator $\Lnt$ and (\ref{Bound}), since $\phi_n(.) \leq 1$, we have

$\begin{aligned}
\Lnt ( X^n_t )^{2p} & = 2p (X^n_t)^{2p-1}\Big(b_n(X^n_t)+f(X^n_t) \Big)+ p(2p-1)\sigma_n^2(X^n_t) (X^n_t)^{2p-2}\\
    & \leq 2p\phi_{n}(X^n_t)\Big(\alpha_p+\beta_p X_t^{2}\Big)(X^n_t)^{2p-2} + 2p (X^n_t)^{2p-2}(X^n_t)f(X^n_t) \\
    & \leq 2p\beta_p (X^n_t)^{2p} +2p\alpha_p (X^n_t)^{2p-2}+ 2p (X^n_t)^{2p-2} \Big( \frac{(X^n_t)^2}{2}+ \frac{f^2(X^n_t)}{2}  \Big)\\
    & \leq p(2\beta_p+1+2k_1^2) (X^n_t)^{2p}+ 2p(\alpha_p+\vert f(0) \vert^2) (X^n_t)^{2p-2} .
 \end{aligned}$\\
  \\
Using It\^o's formula,
\begin{equation*}
\frac{d}{dt} \E\Big[ ( X^n_t )^{2p}\Big]= \E\Big[\Lnt( X^n_t )^{2p}\Big] \leq  \beta'_p   \E\Big[ (X^n_t)^{2p} \Big]+\alpha'_p \E\Big[   (X^n_t)^{2p-2} \Big]
\end{equation*}
Now setting $p=2$ and applying Gronwall's inequality in \cite{Hasminskii12}, for $p=2$ we derive (\ref{ssup}). By induction on $p$ and 
the following interpolation inequality 
\begin{equation*}
\E\Big[ \vert X^n_t \vert^{2p+1} \Big] \leq \Big(\E[(X^n_t)^{2p}]\Big)^{\frac12} \Big(\E[(X^n_t)^{2p+2}]\Big)^{\frac12},
\end{equation*}
we conclude that for every $p \geq 2$ inequality (\ref{ssup}) holds. \\
\end{proof}
\begin{cor}\label{verge}
The following properties hold 
\begin{itemize}
\item  
 For every $p \geq 2$, $X^n_t \longrightarrow X_t$ in $L^p(\Omega)$,
\item 
For every $q \geq q_1$, $b_n(X^n_t) \longrightarrow b(X_t)$ and $\sigma_n(X^n_t) \longrightarrow \sigma(X_t)$in $L^q(\Omega)$.
\end{itemize}
\end{cor}
\begin{proof}
By inequality (\ref{ssup}) and uniform integrability of sequence $\{X_t^n\}$, we derive uniform integrability of sequences  $\{b'_n(X_t^n)\}$  and $\{\sigma'_n(X_t^n)\}$. Also, we know that t $X_t^n$ converges to 
$X_t$ almost surely. Thus $X^n_t \longrightarrow X_t$ in $L^p(\Omega)$ and 
$b'(X^n_t) \longrightarrow b'(X_t)$ in $L^p(\Omega)$. On the other hand, for every $\epsilon \geq 0$ and every integer $p \geq 2$ 
\begin{equation*}
 P(\vert b_n'(X_t^n)- b(X_t^n) \vert > \epsilon) \leq P(\vert X_t^n \vert > n)  \leq \frac{\sup_n \E[\vert X_t^n \vert ^p]}{n^p} \leq \frac{c_p}{n^p}
 \end{equation*}
 So that, $b'_n(X^n_t)-b'(X^n_t) \longrightarrow 0$ in probability and  thus in  $L^p(\Omega)$ by uniform integrability of $\{b'_n(X^n_t)-b'(X^n_t)\}$. The triangle inequality completes the proof.
\end{proof}
\subsection{Weak differentiability of the solution}
In this subsection, we prove the weak differentiability of  $X_t$, using Lemma 1.2.3 in \cite{Nualart06} and Theorem \ref{2lay}.
Then, this fact and Theorem \ref{mlay} results Theorem \ref{Dinf}. \\
\begin{lem}\label{firstderiv}
Assuming Hypothesis {\rm \ref{hypo2}}, for every $p>1$ the unique strong
solution of SDE {\rm (\ref{equa})} is in $\D^{1,p}$. In addition, for $r > t$, $\Drt =0$ and for $r \leq t$ 
\begin{equation}\label{equalinear}
\Drt =\sigma(X_{r})+\int_{r}^{t}\Big( b'(X_{s})+ f'(X_{s})\Big)\Drs
ds+\int_{r}^{t} \sigma'(X_{s})\Drs dW_{s}.
\end{equation}

\end{lem}

\begin{proof}
By Theorem 2.2.1 in \cite{Nualart06} we know that for every $r > t$, $\Dnrt=0$ and for every $r \leq t$ 
\begin{equation*}
\Dnrt =\sigma_n(X^n_{r})+\int_{r}^{t}\Big(\bnps+\fnps\Big)\Dnrs
ds+\int_{r}^{t} \sigma'_n(X_{s}^n)\Dnrs dW_{s}.
\end{equation*}
By Theorem \ref{2lay}, it is sufficient to find a Lyapunov function satisfying conditions (\ref{L}) and (\ref{r}),  which implies that for some constant $c_q$, 
\begin{equation}\label{karand}
\sup_{n \geq 1}\sup_{0 \leq t \leq T}\E\Big[\vert \Dnt \vert^p\Big] \leq c_{q},
\end{equation}
Then we will show that $DX_t$ is the solution of SDE (\ref{equalinear}). To this end, we proceed through the following three steps.\\
{\bf Step 1}. In this step, we introduce some Lyapunov functions $V_q(.)$ for every $q \geq 1$ and obtain some upper bounds for $L_{n,r}V_q$ in terms of $V_q(.)$ and $\sigma'(.)$. \\
For every $q \geq 1$ and every $M > 0$ large enough, choose the following Lyapunov function 
$$V_q(x,y):= x^{4q}+x^{2q}y^{2q} +y^{2q}+M,$$
then by  definition of $V_q$,  equation (\ref{ssup}) and Theorem 2.2.1 in  \cite{Nualart06}, the conditions (\ref{cL}) and (\ref{r}) hold. To prove (\ref{L}), note that by  definition of $L_{n,r}$, 
it holds
\begin{align}
L_{n,r} V_q(x, y) & = 4qx^{4q-2}\Big[ x \Big(b_n(x)+f(x)\Big) + \frac{4q-1}{2} \sigma_n^2(x) \Big]  \nonumber \\
& +2qy^{2q}x^{2q-2}\Big[ x \Big(b_n(x)+f(x)\Big)+ \frac{2q-1}{2} \sigma_n^2(x)\Big] \nonumber \\
& +2q y^{2q}\Big[2qx^{2q-1}\sigma_n(x)\sigma'_n(x) +\frac{2q-1}{2} (\sigma'_n)^2(x)\Big( x^{2q} +1 \Big) \Big]  \nonumber \\
&+ 2q y^{2q}\Big[ \Big(b'_n(x)+f'(x)\Big) \Big( x^{2q} +1 \Big) \Big] =:( I_1+I_2+I_3+I_4)(x,y).   \label{I's}
\end{align}
Since $f'$ is bounded, using (\ref{bnprim}) we get
\begin{equation}\label{I4}
I_4(x,y) \leq  2q(k_1+C_{0,b}) V_{q}(x,y). 
\end{equation}
Applying (\ref{bnprim})

$\begin{aligned}
I_3(x,y) & \leq  2q y^{2q}\Big[2qx^{2q-2}\Big( \frac{(x \sigma'_n(x))^2}{2}+\frac{\sigma_n^2}{2}\Big)+\frac{2q-1}{2} (\sigma'_n)^2(x)\Big( x^{2q} +1 \Big) \Big] \nonumber \\
& \leq 2qy^{2q}x^{2q-2} \Big[ \frac{2q}{2} \sigma^2(x) + (4q-1) x^2\sigma'^2(x) \Big] \nonumber \\
&  +2q(2q-1) y^{2q} \sigma'^2(x)+q(4q-1)C_{0,\sigma}x^{2q}y^{2q}+q(2q-1)C_{0,\sigma}y^{2q}.
\end{aligned}$\\

Now by (\ref{Bound}), we have 
\begin{align}
(I_1+I_2+ I_3)(x,y) & \leq 4qx^{4q-2}\Big[ \alpha_q +\beta_q \vert x \vert^2\Big]+ 4qx^{4q-2}\Big[\frac{x^2}{2}+\frac{f^2(x)}{2} \Big]  \nonumber \\
& +2qy^{2q}x^{2q-2}\Big[\alpha_q +\beta_q \vert x \vert^2\Big]+2qy^{2q}x^{2q-2}\Big[\frac{x^2}{2}+\frac{f^2(x)}{2} \Big]  \nonumber \\
& + 2q(2q-1) y^{2q}(\sigma'(x))^2+q(4q-1)C_{0,\sigma}x^{2q}y^{2q}+q(2q-1)C_{0,\sigma}y^{2q} ~~~~~~~  \label{I1I2I3} 
\end{align}
To show that the terms in the right hand side of the latter inequality are bounded, we need the Young's inequality.
Using (\ref{young}) for $r=4q$ and again for $r=2q$ with $c=\triangle_1=1$, we have 
\begin{equation}\label{x4q2q}
x^{4q-2} \leq \frac{2q-1}{2q}x^{4q}+\frac{1}{2q},  \qquad  \textmd{and} \qquad  x^{2q-2} \leq \frac{q-1}{q}x^{2q}+\frac{1}{q}.
\end{equation}
Applying (\ref{x4q2q}) in (\ref{I1I2I3}), we find some constant $c_q$ such that 
\begin{equation}\label{1}
I_1(x,y)+I_2(x,y)+ I_3(x,y)  \leq c_q V_q(x,y) + 2q(2q-1) y^{2q}(\sigma'(x))^2.  
\end{equation}
Substituting (\ref{I4}) and (\ref{1}) in (\ref{I's}), we conclude that 
\begin{equation}\label{LVG}
L_{n,r} V_q(x, y)  \leq (c_q+2q(k_1+C_{0,b}) V_q(x,y) + 2q(2q-1) y^{2q}(\sigma'(x))^2  
\end{equation}
{\bf Step 2}. Here, we show the inequality (\ref{L}) for every $q \geq 1$. First, let  $q=q_1$. By (\ref{polyno}) for some  $c_q$ independent of $n$, we have \\

$\begin{aligned}
 2q(2q-1) y^{2q}(\sigma'(x))^2  & \leq  \gamma_1 4q(2q-1) y^{2q}+\gamma_1 4q(2q-1) y^{2q}\vert x \vert^{2q_1} \leq c_{q_1} V_{q_1}(x,y).     \\
\end{aligned}$\\

Substitute this bound in (\ref{LVG}) and use Theorem \ref{2lay} to derive (\ref{karand}) for $p=q_1$, from which and interpolation inequality we derive the result for each $p < q_1$.\\
Now, let $q > q_1$. Using H\"older inequality, for some constant $c'_{q_1}$ independent of $n$ we have \\

$\begin{aligned}
 y^{2q}x^{2q_1} = y^{2q-2q_1}x^{2q_1}y^{2q_1}  & \leq  \frac{q-q_1}{q} \Big(y^{2q-2q_1}\Big)^{\frac{q}{q-q_1}}+ \frac{q_1}{q}\Big( x^{2q_1}y^{2q_1}\Big)^{\frac{q}{q_1}} \nonumber \\
 & \leq  \frac{q-q_1}{q} y^{2q }  +   \frac{q_1}{q} x^{2q}y^{2q} \leq c'_{q_1} V_{q}(x,y)\\
\end{aligned}$\\
\\
Again, substitute this bound in (\ref{LVG}) and use Theorem \ref{2lay} to derive (\ref{karand}) for $p >q_1$.\\
{\bf Step 3}. Now, we show that $DX_t$ is the solution of SDE (\ref{equalinear}). For every $\epsilon > 0$, by Lemma (\ref{io}) we have 
\begin{equation*}
P(\vert \Dnrt - \Drt \vert > \epsilon) \leq P( t \geq \tau_n ) 
                       \leq P(\sup_{0 \leq s \leq t} \vert X_s \vert > n) \leq \frac{\E\Big[\sup_{0 \leq t \leq s} \vert X_s^n \vert^p\Big]}{n^p} \leq \frac{c'_p}{n^p}
\end{equation*}
Therefore, $\Dnrt \longrightarrow \Drt$ in probability. Since the sequence $\vert \Dnrt \vert ^p$ is uniformly integrable, this convergence still hold in $L^p(\Omega)$ for every $p \geq 2$. From Corollary \ref{verge},  $\Drt$ is the solution to SDE (\ref{equalinear}) and the proof is completed.
\end{proof}
Now, Since the Lyapunov function $V_q$ and the functions $b$ and $\sigma$ have polynomial growth, by induction on $m$ in Theorem \ref{mlay} we can derive  inequalities (\ref{rm}) and (\ref{xym}). Condition (\ref{cLm})  is also obviously true, and Theorem \ref{Dinf} follows as a result. 
\section{The nondegeneracy condition}
In this section, we will show how the regularity of the distribution of  $X(t)$ could be derived from the nondegeneracy condition of  it. \\
Denote the Malliavin covariance matrix of $X_t^n$ and $X_t$  by $\Lambda_{X^n}(t)$ and $\Lambda_{X}(t)$, for each $0 \leq t \leq T$, respectively.
Let $Z^n_t$ be the solution of the following linear SDE;
$$Z^n_t= 1+ \int_{0}^{t} [b'_n(X^n_s) +f'(X^n_s)]Z^n_s ds + \int_{0}^{t}  \sigma'_n(X^n_s) Z^n_s dW_s $$
and
\begin{equation*}
C^n_t:= \int_{0}^{t} \exp\Big\{-2\int_{0}^{r} [b_n'(X_s^n)-\frac12 (\sigma'_n)^2(X_s^n)]ds+2\int_{0}^{r} \sigma'_n(X_s^n)]dW_s\Big\} \sigma^2(X_r^n) dr. 
\end{equation*}
Then $\Lambda_{X^n}(t)= C_t^n(Z^n_t)^{-2}$. Also, by the proof of Theorem \ref{firstderiv}, one can easily show that for every $p \geq 2$, there exist some constant $l_p$ such that 
\begin{equation}\label{lz}
\sup_{n \geq 1}\E\Big[\vert Z_t^n \vert^p \Big] \leq l_p.
\end{equation}
\begin{lem}
The nondegeneracy condition  is satisfied for $X_t$, and for every $p \geq 2$ and $\epsilon < \epsilon_0(p)$, 
\begin{equation*}
P\Big(\Lambda_{X}(t) \leq \epsilon\Big) \leq \epsilon^p
\end{equation*}
\end{lem}
\begin{proof}
By  definition of $b_n$ and $\sigma_n$, it is easy to derive  condition {\bf (H)} for these coefficients. So, the nondegeneracy condition is satisfied for $X_t^n$ for every $n \geq 1$. From (\ref{lz}) and Proposition \ref{horman}, for every $n \geq 1$ and small enough $\epsilon$,
\begin{equation*}
P\Big(\Lambda_{X^n}(t) \leq \epsilon\Big) \leq P\Big(C^n_t \leq \epsilon\Big)+ P\Big((Z_t^n)^{-2} < \epsilon \Big) \leq \epsilon^p\Big(1+\E\Big[(Z^n_t)^p\Big]\Big).
\end{equation*}
Also, \\

$\begin{aligned}
\{\Lambda_{X}(t) \leq \epsilon\} &\subseteq \{\Lambda_{X^n}(t) \leq 2\epsilon\}\cup \{\Lambda_{X^n}(t) > 2\epsilon \quad \textmd{and} \quad \Lambda_{X}(t) \leq \epsilon\} \\
     &    \subseteq \{\Lambda_{X^n}(t) \leq 2\epsilon\}\cup \{\Lambda_{X^n}(t) \leq 2\epsilon \quad \textmd{and} \quad \vert \Lambda_{X^n}(t) - \Lambda_{X}(t) \vert > \epsilon\}   \\
                   &    \subseteq \{\Lambda_{X^n}(t) \leq 2\epsilon\}\cup \{\vert \Lambda_{X^n}(t) - \Lambda_{X}(t) \vert > \epsilon\}   \\
\end{aligned}$\\
Now, by  Step3 in the proof of Theorem \ref{firstderiv}, we can choose $N \geq 1$ such that for every $n \geq N$,
$$\E\Big[\vert \Dnrt - \Drt \vert^2\Big] < \epsilon^{p+1}.$$
Then, \\
$$P\Big(\Lambda_{X}(t) \leq \epsilon \Big) \leq P\Big({\Lambda}_{X^N}(t) \leq 2\epsilon\Big)+P\Big(\vert \Lambda_{X^N}(t) - \Lambda_{X}(t) \vert > \epsilon \Big)\\$$ 
 $$ \leq  (1+l_p){\epsilon}^p+  \frac{1}{\epsilon} \E\Big[\vert \Lambda_{X^N}(t) - \Lambda_{X}(t) \vert^2\Big]$$             
$$        \leq  (1+l_p){\epsilon}^p + (t-r){\epsilon}^p,  $$
and the nondegeneracy condition holds for $X_t$.
\end{proof}

Therefore, we have a nondegenerate solution $X_t$ to  SDE (\ref{equa}) which has a $C^{\infty}$-density and thus Theorem \ref{DENSF} hold. 
\appendix
\section{constructing the approximating functions for the drift}
In this Appendix we present how we choose the functions $b_n$. This construction is motivated by Berhanu in \cite[Theorem 2.9.]{Berhanu}. Assume that $U \subset V$ be two open sets in $\R^d$ with distance $a$. For $0 \leq \epsilon \leq a$, define 
$U_{\epsilon}=\{x ; d(x, U) < \epsilon \}$. Then $U_\epsilon = \bigcup_{x \in U} B_\epsilon(x)$ and $U \subseteq U_\epsilon \subseteq V$. Fix $\epsilon$ such that 
$0 < 2\epsilon \leq a$ and let $h^\epsilon(x) =$ the characteristic function of $u_\epsilon$. Let $\psi \in C_0^\infty(\R^d)$ with $supp \psi \subseteq  B_1(0)$ and
$\int \psi(x) dx = 1$. Set $\psi_\epsilon(x) =
\frac{1}{\epsilon^{d}}\psi(\frac{x}{\epsilon})$. Consider now the convolution 
function $\psi_\epsilon \star h^\epsilon$ for $0 < 2\epsilon < d$. Since $supp \psi_\epsilon \subseteq B_\epsilon(0)$, $\psi_\epsilon \star h^\epsilon=1$ on $U$ and 
$\psi_\epsilon \star h^\epsilon=0$ outside $U_{2\epsilon}$. Note that 
for each multiindex $\alpha$,
\begin{align}
\partial_\alpha (\psi_\epsilon \star h^\epsilon) (x) & = \int \partial_\alpha ( \psi_\epsilon(y)) h^\epsilon(x-y) dy  = \frac{1}{\epsilon^{d+\vert \alpha \vert}}\int (\partial_\alpha \psi)(\frac{y}{\epsilon}) h^\epsilon(x-y) dy \nonumber \\
& =  \frac{1}{\epsilon^{\vert \alpha \vert}}\int (\partial_\alpha \psi)(z) h^\epsilon(x-\epsilon z) dz \leq \parallel \psi\parallel_\infty \frac{1}{\epsilon^{\vert \alpha \vert}}   \label{consphin}
\end{align}
Now, for every $n \geq 1$ consider $U=B_{n^\xi}(0)$, $V=B_{2n^\xi}(0)$ and $\epsilon=n^\xi $. Then thete exist the functions $\phi_n(x) := \psi_{\epsilon} \star h^\epsilon$ such that $\phi_n(x)=1$ on $U$ and $\phi_n(x)=0$ outside $V$. Since $supp \phi_n(x) \subseteq B_{2n^\xi}(0)$, by (\ref{consphin}) and (\ref{polyno}) for each multiindex $\alpha$ with $\vert \alpha \vert=c \geq 1$, we have \\

$\begin{aligned}
\vert b(x) \partial_\alpha \phi_n(x) \vert & \leq \vert b(x)\chi_{\vert x \vert \leq 2n^\xi} \vert \parallel \psi\parallel_\infty \frac{1}{n^{\xi \vert \alpha \vert}} \\
& \leq \gamma_c (1+ 2^{\xi} n^\xi)\parallel \psi\parallel_\infty \frac{1}{n^{\xi \vert \alpha \vert}} \leq 2^{\xi+1}\gamma_c \parallel \psi\parallel_\infty.
\end{aligned}$\\
In the same way, it hold true when we replace $b$ by $\sigma$. Also, 
\begin{equation*}
\vert \partial_\alpha \phi_n(x) \vert \leq \parallel \psi\parallel_\infty. 
\end{equation*}
\\


\end{document}